\documentclass[11pt]{amsart}
\usepackage{amsmath, amssymb}
\usepackage{amsfonts}
\usepackage{graphicx}
\usepackage{mathrsfs}
\usepackage[arrow,matrix,curve,cmtip,ps]{xy}
\usepackage{amsthm}
\usepackage{enumerate}
\usepackage{color}
\usepackage[colorlinks]{hyperref}

\allowdisplaybreaks

\newtheorem{theorem}{Theorem}[section]
\newtheorem{lemma}[theorem]{Lemma}
\newtheorem{proposition}[theorem]{Proposition}
\newtheorem{corollary}[theorem]{Corollary}

\newtheorem*{theorem*}{Theorem}
\theoremstyle{remark}
\newtheorem{remark}[theorem]{Remark}
\newtheorem{definition}[theorem]{Definition}
\newtheorem{example}[theorem]{Example}

\numberwithin{equation}{section}

\newcommand{\Z}{\mathbb{Z}}
\newcommand{\N}{\mathbb{N}}

\newcommand{\lr}{L_R(E)}

\begin{document}
\title[Ideal structure of Leavitt path algebras]{Ideal structure of Leavitt path algebras with coefficients in a unital commutative ring}

\author[H. Larki]{Hossein Larki}

\address{Department of Pure Mathematics\\ Faculty of Mathematics
and Computer Science\\ Amirkabir University of Technology, No. 424, Hafez
Ave.,15914\\ Tehran, Iran}
\email{h.larki@aut.ac.ir}


\date{\today}



\begin{abstract}
For a (countable) graph $E$ and a unital commutative ring $R$, we analyze the ideal structure of the Leavitt path algebra $\lr$ introduced by Mark Tomforde. We first modify the definition of basic ideals and then develop the ideal characterization of Mark Tomforde. We also give necessary and sufficient conditions for the primeness and the primitivity of $\lr$, and we then determine prime graded basic ideals and left (or right) primitive graded ideals of $\lr$. In particular, when $E$ satisfies Condition (K) and $R$ is a field, they imply that the set of prime ideals and the set of primitive ideals of $\lr$ coincide.
\end{abstract}

\maketitle

\section{Introduction}

There have been recently a great deal of interest in generalization of the Leavitt path algebras associated to directed graphs \cite{abr2,tom2,tom1}. Given a row-finite graph $E$ and a field $K$, in \cite{abr1,ara4} Leavitt path algebra $L_K(E)$ was defined by an algebraic analogue of the definition of graph $C^*$-algebra $C^*(E)$ described in \cite{rae1}. After that in \cite{abr2} this definition was extended for every countable graph $E$. The algebras $L_K(E)$ are generalizations of the algebras $L(1,n)$, $n\geq 2$, investigated by Leavitt \cite{lea}, and are a specific type of path $K$-algebras associated to a graph (modulo certain relations). This class of algebras is important for two reasons. First, some examples of many interesting classes of algebras can be seen as the Leavitt path algebras of some specific graphs rather than the classical Leavitt algebras $L(1,n)$, such as: (1) matrix rings $M_n(K)$ for $n\in \N\cup\{\infty\}$ (where $M_\infty(K)$ denotes $\N\times \N$ matrices with only a finite number of nonzero entries), (2) the Toeplitz algebra, and (3) the Laurent polynomials ring $K[x,x^{-1}]$. Second, they are useful because structural properties of the algebras can be related to simple observations about their underlying graphs. Concretely, the literature on Leavitt path algebras includes necessary and sufficient conditions on a directed graph $E$ so that the corresponding Leavitt path algebra $L_K(E)$ is simple \cite{abr1,tom1}, purely infinite simple \cite{abr2}, finite dimensional \cite{abr4}, locally finite \cite{abr5}, and semisimple \cite{abr3}.

Recently, Mark Tomforde introduced Leavitt path algebras $L_R(E)$ with coefficients in a unital commutative ring $R$ \cite{tom1} and proved two well-known uniqueness theorems for this class of algebras; namely, the Graded Uniqueness Theorem and the Cuntz-Krieger Uniqueness Theorem. He then defined basic ideals, and in the case that $E$ is a row-finite graph, characterized graded basic ideals of $L_R(E)$ by the saturated hereditary subsets of vertices in $E$. For non-row finite graphs, the set of admissible pairs in $E$ are considered instead of saturated hereditary subsets of $E^0$. However, in this case, basic ideals of $L_R(E)$ are more complicated, and to determined a one-to-one correspondence between graded basic ideals of $L_R(E)$ and admissible pairs in $E$, we need to modify the definition of basic ideals. We give a new definition of basic ideal such that when $E$ is row-finite, it is equivalent to that in \cite{tom1}. Then the ideal characterization of \cite{tom1} is developed for every countable graph.

In \cite{hon}, the primitive spectrum of a graph $C^*$-algebra $C^*(E)$ was characterized. By inspiring the ideas of \cite{hon}, Aranda Pino, Pardo, and Siles Molina determined the prime spectrum and primitivity of a Leavitt path algebras $\lr$ in \cite{ara3} when $E$ is row-finite and $R$ is a field. The knowledge of the prime and primitive Leavitt path algebras is a fundamental and necessary step towards the classification of these algebras. In particular, the prime spectrum for commutative rings is to carry information over from Algebra to Topology and vice versa. Also, we know that the primitive ideals of a ring naturally correspond to its irreducible representations. In this paper, we determine prime and primitive Leavitt path algebras by giving necessary and sufficient conditions for the coefficients ring $R$ and the underlying graph $E$. These results are a generalization of results in \cite{ara3} which were proved for the case that $E$ is row-finite and $R$ is a field. Moreover, by applying these, we give graph theoretic descriptions of prime and primitive graded basic ideals of $\lr$.

The present article is organized as follows. We begin by Section 2 to provide some basic facts about Leavitt path algebras. Most of our definitions in this section are from \cite{tom1}. For a directed graph $E$ and a unital commutative ring $R$, it is associated an $R$-algebra $L_R(E)$. These algebras are defined as the definition of graph $C^*$-algebras $C^*(E)$ and they have natural $\Z$-grading. In Section 3, we study ideal structure of $L_R(E)$. Our characterization are inspired by the structure of graded ideals of $L_K(E)$ in \cite{tom2}, and graded basic ideals of $L_R(E)$ in \cite{tom1} when $E$ is row-finite. We develop the ideal characterization of \cite{tom1} for non-row-finite cases. In spite of what claimed in \cite{tom1}, to do this we need to change the definition of basic ideal in \cite{tom1} so that graded basic ideals of $L_R(E)$ are corresponding to admissible pairs in $E$. However, our definition is equivalent to \cite[Definition 7.2]{tom1} when underlying graph is row-finite. Then graded basic ideals of $L_R(E)$ are determined by admissible pairs in $E$.

In Section 4, we give two necessary and sufficient conditions for the primeness of a Leavitt path algebra $\lr$ and then we apply this result to determine prime graded basic ideals. In particular, when $R$ is a field and $E$ satisfies Condition (K), our results determine the prime spectrum of $\lr$. In Section 5, the minimal left ideals and the minimal right ideals of $\lr$ generated by a single vertex  are characterized. We show that a left ideal $\lr v$ is minimal in $\lr$ if and only if $R$ is a field and $v$ is a line point. This is a natural generalization of the results in \cite[Section 2]{ara1} which will be needed in Section 6.

Finally, in Section 6, we give two necessary and sufficient conditions for the left and right primitivity of $\lr$ as a ring. We then characterize primitive graded ideals of $\lr$. In particular, when $R$ is a field and $E$ satisfies Condition (K), all primitive ideals of $\lr$ are characterized and we show that the set of primitive ideals and the set of prime ideals of $\lr$ coincide.

\section{Basic facts about Leavitt path algebras}

In this section we provide the basic definitions and properties of Leavitt path algebras which will be used in the next sections. For more thorough introduction, we refer the reader to \cite{tom2,tom1}.

\begin{definition} A (directed) graph $E=(E^0,E^1,r,s)$ consists of a countable set of vertices $E^0$, a countable set of edges $E^1$, a source function $s:E^1\rightarrow E^0$, and a range function $r:E^1\rightarrow E^0$. A vertex $v\in E^0$ is called a \emph{sink} if $s^{-1}(v)=\emptyset$, is called \emph{finite emitter} if $0<|s^{-1}(v)|<\infty$, and is called \emph{infinite emitter} if $|s^{-1}(v)|=\infty$. If $s^{-1}(v)$ is a finite set for every $v\in E^0$, then $E$ is called \emph{row-finite}.
\end{definition}
If $e_1,\ldots,e_n$ are edges such that $r(e_i)=s(e_{i+1})$ for $1\leq i< n$, then $\alpha=e_1\ldots e_n$ is called a \emph{path} of length $|\alpha|=n$ with source $s(\alpha)=s(e_1)$ and range $r(\alpha)=r(e_n)$. For $n\geq 2$, we define $E^n$ to be the set of paths of length $n$, and $\mathrm{Path}(E):=\bigcup_{n=0}^\infty E^n$ the set of all finite paths. Note that we consider the vertices in $E^0$ to be paths of length zero.

A \emph{closed path based at $v$} is a path $\alpha\in \mathrm{Path}(E)\setminus E^0$ such that $v=s(\alpha)=r(\alpha)$. A closed path $\alpha=e_1\ldots e_n$ is called a \emph{cycle} if $s(e_i)\neq s(e_j)$ for $i\neq j$. We say that a closed path $\alpha=e_1\ldots e_n$ has an \emph{exit} if there is a vertex $v=s(e_i)$ and an edge $f\in s^{-1}(v)\setminus \{e_i\}$. If  $s(\alpha)=r(\alpha)$ and $s(e_i)\neq s(e_1)$ for every $i>1$, then $\alpha$ is called a \emph{closed simple path}.

\begin{definition}\label{defn1.3}
Let $(E^1)^*$ denote the set of formal symbols $\{e^*:e\in E^1\}$. We define $v^*=v$ for all $v\in E^0$, and for a path $\alpha=e_1\ldots e_n\in E^n$ we define $\alpha^*:=e_n^*\ldots e_1^*$. The elements of $E^1$ are said \emph{real edges} and the elements of $(E^1)^*$ are said \emph{ghost edges}.
\end{definition}

\begin{definition}\label{defn1.4}
Let $E$ be a graph and let $R$ be a ring. A Leavitt $E$-family is a set $\{v,e,e^*: v\in E^0,e\in E^1\}\subseteq  R$ such that $\{v:v\in E^0\}$ consists of pairwise orthogonal idempotents and the following conditions are satisfied:
\begin{enumerate}
\item $s(e)e=er(e)=e$ for all $e\in E^1$,
\item $r(e)e^*=e^*s(e)=e^*$ for all $e\in E^1$,
\item $e^*f=\delta_{e,f}r(e)$ for all $e,f\in E^1$, and
\item $v=\sum_{s(e)=v}ee^*$ whenever $0<|s^{-1}(v)|<\infty$.
\end{enumerate}
\end{definition}

\begin{definition}
Let $E$ be a graph and let $R$ be a unital commutative ring. The \emph{Leavitt path algebra of $E$ with coefficients in $R$}, denoted by $L_R(E)$, is the universal $R$-algebra generated by a Leavitt $E$-family.
\end{definition}

The universal property of $L_R(E)$ means that if $A$ is an $R$-algebra and $\{a_v,b_e,b_{e^*}:v\in E^0, e\in E^1\}$ is a Leavitt $E$-family in $A$, then there exists an $R$-algebra homomorphism $\phi:L_R(E)\rightarrow A$ such that $\phi(v)=a_v,~\phi(e)=b_e,$ and $\phi(e^*)=b_{e^*}$ for all $v\in E^0$ and $e\in E^1$.

When $R$ is a field, the existence of such universal algebra $L_R(E)$ was shown in \cite{abr2} (also see \cite[Remark 2.6]{tom2}). A similar argument shows that $L_R(E)$ exists when $R$ is a unital commutative ring (\cite[Remark 3.6]{tom1}). By \cite[Proposition 3.4]{tom1} we see that
\begin{equation}\label{2.1}
L_R(E)=\mathrm{span}_R\{\alpha\beta^*:\alpha,\beta\in \mathrm{Path}(E) ~\mathrm{and}~ r(\alpha)=r(\beta)\}
\end{equation}
and $rv\neq 0$ for all $v\in E^0$ and $r\in R\setminus\{0\}$. This implies that $r\alpha\beta^*\neq 0$ for all $r\in R\setminus\{0\}$ and paths $\alpha,\beta\in \mathrm{Path}(E)$ with $r(\alpha)=r(\beta)$, because $\alpha^*(r\alpha\beta^*)\beta=rr(\alpha)\neq 0$.

Recall that a \emph{set of local units} for a ring $R$ is a set $U\subseteq  R$ of commuting idempotents with the property that for any $x\in R$ there exists $u\in U$ such that $ux=xu=x$. If $E^0$ is finite, then $1=\sum_{v\in E^0}v$ is a unit for $L_R(E)$. If $E^0$ is infinite, then $L_R(E)$ does not have a unit, but if we list the vertices of $E$ as $E^0=\{v_1,v_2,\ldots\}$ and set $t_n:=\sum_{i=1}^nv_i$, then Equation \ref{2.1} implies that $\{t_n\}_{n\in\mathbb{N}}$ is a set of local units for $L_R(E)$.

Leavitt path algebras have a $\Z$-grading which is important to investigate. This property plays an action similar to the gauge action for graph $C^*$-algebras.

\begin{definition}
A ring $R$ is called $\mathbb{Z}$-\emph{graded} (or, more concisely, \emph{graded}) if there is a collection of additive subgroups $\{R_k\}_{k\in\mathbb{Z}}$ of $R$ such that
\begin{enumerate}
\item $R=\bigoplus_{k\in\mathbb{Z}}R_k$
\item $R_jR_k\subseteq  R_{j+k}$ for all $j,k\in\mathbb{Z}$.
\end{enumerate}
The subgroup $R_k$ is said the \emph{homogeneous component of $R$ of degree $k$}. In this case an ideal $I$ of $R$ is called a \emph{graded ideal} if $I=\bigoplus_{k\in\mathbb{Z}}(I\cap R_k)$. Furthermore, if $\phi:R\rightarrow S$ is a ring homomorphism between graded rings, then $\phi$ is a \emph{graded ring homomorphism (of degree zero)} if $\phi(R_k)\subseteq S_k$ for all $k\in\mathbb{Z}$.
\end{definition}

Note that the kernel of a graded homomorphism is a graded ideal. Also, if $I$ is a graded ideal of a ring $R$, then the quotient $R/I$ is naturally graded with homogeneous components $\{R_k+I\}_{k\in\mathbb{Z}}$ and the quotient map $R\rightarrow R/I$ is a graded homomorphism. If $E$ is a graph and $R$ is a unital commutative ring, then \cite[Proposition 4.7]{tom1} implies that the Leavitt path algebra $L_R(E)$ is graded with the natural homogeneous components
$$L_R(E)_k:=\mathrm{span}_R\left\{\alpha\beta^*:\alpha,\beta\in \mathrm{Path}(E),~\mathrm{and}~|\alpha|-|\beta|=k\right\}.$$

\section{Ideals of Leavitt path algebras}

In this section a special class of graded ideals in Leavitt path algebras, namely graded basic ideals, is characterized. These results will be a generalization of both the results of \cite[Sections 5 and 6]{tom2} which were for Leavitt path algebras with coefficients in a field and the results of \cite[Section 7]{tom1} which were for Leavitt path algebras associated to row-finite graphs. This characterization will be done by admissible pairs in $E$.

Let $E$ be a graph. Recall that a subset $X\subseteq E^0$ is called \emph{hereditary} if for any edge $e\in E^1$ with $s(e)\in X$ we have that $r(e)\in X$. Also, we say that $X\subseteq E^0$ is \emph{saturated} if for each finite emitter $v\in E^0$ with $r(s^{-1}(v))\subseteq X$ we have $v\in X$. The saturation of a subset $X\subseteq E^0$, denoted by $\overline{X}$, is the smallest saturated subset of $E^0$ containing $X$.

Observe that the intersections of saturated hereditary subsets are saturated and hereditary. Also, unions of saturated hereditary subsets are hereditary, but not necessarily saturated. Note that if $X$ is hereditary, then $\overline{X}$ is hereditary.

\begin{definition}
Suppose that $H$ is a saturated hereditary subset of $E^0$. The set of vertices in $E^0$ which emit infinitely many edges into $H$ and finitely many into $E^0\setminus H$ is denoted by $B_H$; that is
$$B_H:=\left\{v\in E^0\setminus H:|s^{-1}(v)|=\infty~\mathrm{and}~0<|s^{-1}(v)\cap r^{-1}(E^0\setminus H)|<\infty\right\}.$$
Also, for any $v\in B_H$ we denote
$$v^H:=v-\sum_{\substack{s(e)=v\\ r(e)\notin H}}ee^*.$$
\end{definition}

\begin{definition}\label{defn3.2}
If $H$ is a saturated hereditary subset of $E^0$ and $S\subseteq B_H$, then $(H,S)$ is called an \emph{admissible pair} in $E$.
\end{definition}

If $(H,S)$ is an admissible pair in $E$, we denote by $I(H,S)$ the ideal of $L_R(E)$ generated by $\{v:v\in H\}\cup\{v^H:v\in S\}$.

\begin{lemma}\label{lem3.3}
Let $E$ be a graph and let $R$ be a unital commutative ring. If $(H,S)$ is an admissible pair in $E$, then
$$I(H,S)=\mathrm{span}_R\left(\{\alpha\beta^*:r(\alpha)=r(\beta)\in H\}\cup\{\alpha v^H\beta^*:r(\alpha)=r(\beta)=v\in S\}\right)$$
and $I(H,S)$ is a self-adjoint graded ideal of $L_R(E)$.
\end{lemma}

\begin{proof}
Let $J$ denote the right-hand side of the above equation. For $\alpha,\beta,\mu,\nu\in E^*$ with $r(\alpha)=r(\beta)\in H$ and $r(\mu)=r(\nu)=v \in S$, we have $\alpha\beta^*=\alpha r(\alpha)\beta^*\in I(H,S)$ and $\mu v^H \beta^*\in I(H,S)$. Hence, $J\subseteq I(H,S)$. The converse follows from the fact that $J$ contains the generators of $I(H,S)$, and consequently, $I(H,S)=J$.

Since $\alpha\beta^*$ and $\alpha v^H\beta^*$ are homogenous of degree $|\alpha|-|\beta|$, $\overline{\alpha\beta^*}=\beta\alpha^*$, and $\overline{\alpha v^H\beta^*}=\beta v^H\alpha^*$, we see that $I(H,S)$ is graded and self-adjoint.
\end{proof}

When $R$ is a field, \cite[Theorem 5.7(1)]{tom2} yields that there are a one-to-one correspondence between graded ideals of $L_R(E)$ admissible pairs in $E$. This result is analogous to \cite[Theorem 3.6]{bat1} for graph $C^*$-algebras. If $R$ is a unital commutative ring, the graded ideals of the Leavitt path algebras $L_R(E)$ are more complicated (see \cite[Example 7.1]{tom1}).

In the case that $E$ is a row-finite graph, for every saturated hereditary subset $H$ of $E^0$ we have that $B_H=\emptyset$. Hence every admissible pair in $E$ is of the form $(H,\emptyset)$ for a saturated hereditary subset $H$. In \cite{tom1}, the author defined basic ideals of $\lr$; an ideal $I$ is called basic if $rv\in I$ implies $v\in I$ for $v\in E^0$ and $r\in R\setminus\{0\}$. Then there was shown that the lattice of graded basic ideals of $\lr$ is isomorphic to the lattice of saturated hereditary subsets of $E^0$.

In spite of what claimed in \cite{tom1}, when $E$ is a non-row-finite graph, the set of graded basic ideals of $\lr$ may not be corresponding to the set of admissible pairs in $E$ (see Example \ref{ex3.4}). In Definition \ref{defn3.5}, we modify the definition of basic ideal and then we show that graded basic ideals of $\lr$ are one-to-one corresponding to the admissible pairs in $E$. However, when $E$ is row-finite our definition of basic ideal is equivalent to its definition in \cite{tom1}.

\begin{example}\label{ex3.4}
Suppose that $E$ is the graph
$$\includegraphics[width=4cm]{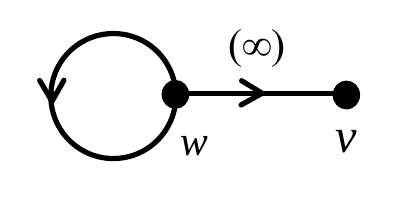}$$
where the symbol ($\infty$) indicates infinitely many edges from $w$ to $v$. Then $H=\{v\}$ is a saturated hereditary subset of $E^0$ and $B_H=\{w\}$. Let $I$ be the ideal of $L_{\mathbb{Z}}(E)$ generated by $\{v,2w^H\}$. Similar to the proof of Lemma \ref{lem3.3}, we have
$$I=\mathrm{span}_{\mathbb{Z}}\left(\{\alpha\beta^*:r(\alpha)=r(\beta)=v\}\cup\{\alpha(2w^H)\beta^*:r(\alpha)=r(\beta)=w\}\right),$$
and hence $I$ is graded. Furthermore, $I$ is a basic ideal in the sense of \cite[Definition 7.2]{tom1}, because the only nonzero multiples of elements in $\Z$ at a vertex in $I$ are of the form of $mv$ for $m\in\mathbb{Z}$. However, $I$ is not of the form $I(H,S)$ for an admissible pair $(H,S)$ in $E$. Indeed, $H:=\{x\in E^0:x\in I\}=\{v\}$ and $S:=\{x\in B_H:x^H\in I\}=\emptyset$, and also we have $I\neq I(H,\emptyset)$ (because $2w^H\in I\setminus I(H,\emptyset)$).
\end{example}

This problem occurs because an ideal $I(H,S)$ is generated not only by the vertices in $H$ but also by $v^H$s for $v\in S$. While, when $E$ is row-finite, we have $B_H=\emptyset$. However, by a small modification in the definition of basic ideal in \cite[Definition 7.2]{tom1}, we can solve the problem. We redefine a basic ideal such that every graded basic ideal of $L_R(E)$ will be generated by an admissible pair $(H,S)$ in $E$.

\begin{definition}\label{defn3.5}
Let $E$ be a graph and let $R$ be a unital commutative ring. An ideal $I$ of $L_R(E)$ is called \emph{basic} if $rx\in I$ implies $x\in I$, where either $x\in E^0$ or $x$ is of the form $x=v-\sum_{i=1}^ne_ie_i^*$ for $v\in E^0$ and $s(e_i)=v$ ($1\leq i\leq n$).
\end{definition}

Note that when $E$ is a row-finite graph, our definition of basic ideal is equivalent to \cite[Definition 7.2]{tom1}. (This is obtained by comparing Theorem \ref{thm3.10}(4) with \cite[Theorem 7.9(1)]{tom1}.) Also, suppose that $I$ is a basic ideal of a Leavitt path algebra $L_R(E)$ and set $H:=I\cap E^0$. If $rv^H\in I$ for some $v\in B_H$ and some $r\in R\setminus\{0\}$, then we have $v^H\in I$.

As for graph $C^*$-algebras, to characterize ideals and quotients of a Leavitt path algebra $L_R(E)$ we usually use some special graphs. We introduced two of them in the following definition. In Definition \ref{defn3.12}, we will define another one.

\begin{definition}\label{defn3.6}
Let $E$ be a graph and let $(H,S)$ be an admissible pair in $E$.
\begin{enumerate}
\item The quotient of $E$ by $(H,S)$ is the graph $E/(H,S)$ defined by\\
   $(E/(H,S))^0:=(E^0\setminus H)\cup \{v': v\in B_H\setminus S\}$\\
   $(E/(H,S))^1:=\{e\in E^1:r(e)\in E^0\setminus H\}\cup\{e':e\in E^1,r(e)\in B_H\setminus S\}$\\
  and $r,s$ are extended to $(E/(H,S))^0$ by setting $s(e')=s(e)$ and $r(e')=r(e)'$.
\item We denote $E_{(H,S)}$ the graph defined by\\
   $E^0_{(H,S)}:=H\cup S$\\
   $E^1_{(H,S)}:=\{e\in E^1:s(e)\in H\}\cup \{e\in E^1:s(e)\in S ~\mathrm{and}~ r(e)\in H\}$\\
   and we restrict $r$ and $s$ to $E^1_{(H,S)}$.
\end{enumerate}
\end{definition}

A similar argument as in the first paragraph of the proof of \cite[Theorem 5.7]{tom2} gives the following lemma.

\begin{lemma}\label{lem3.7}
Let $E$ be a graph and let $R$ be a unital commutative ring. If $(H,S)$ is an admissible pair in $E$ and $I(H,S)$ is the ideal of $L_R(E)$ defined in Definition \ref{defn3.2}, then $\{v\in E^0:v\in I(H,S)\}=H$ and $\{v\in B_H: v^H\in I(H,S)\}=S$.
\end{lemma}

\begin{lemma}[\cite{tom1}, Lemma 7.8]\label{lem3.8}
Let $E$ be a graph and let $R$ be a unital commutative ring. If $X$ is a hereditary subset of $E^0$ and $I(X)$ is the ideal of $L_R(E)$ generated by $\{v: v\in X\}$, then $I(X)=I(\overline{X},\emptyset)$.
\end{lemma}

\begin{lemma}\label{lem3.9}
Let $E$ be a graph and let $R$ be a unital commutative ring. If $(H,S)$ is an admissible pair in $E$, then $I(H,S)$ is a basic ideal.
\end{lemma}

\begin{proof}
We first show that $rv\in I(H,S)$ implies $v\in I(H,S)$ and $rw^H\in I(H,S)$ implies $w\in S$, for $v\in E^0$, $w\in B_H$, and $r\in R\setminus\{0\}$. For this, consider the graph $E/(H,S)$ in Definition \ref{defn3.6} and its associated Leavitt path algebra $L_R(E/(H,S))$. For each $v\in E^0$ and $e\in E^1$ define
$$a_v:=\left\{
        \begin{array}{ll}
          v & \mathrm{if~}v\in(E^0\setminus H)\setminus(B_H\setminus S)  \\
          v+v' & \mathrm{if~} v\in B_H\setminus S \\
          0 & \mathrm{if~} v\in H,
        \end{array}
      \right.
$$
$$b_e:=\left\{
        \begin{array}{ll}
          e & \mathrm{if~}r(e)\in(E^0\setminus H)\setminus(B_H\setminus S)  \\
          e+e' & \mathrm{if~} r(e)\in B_H\setminus S \\
          0 & \mathrm{if~} r(e)\in H,
        \end{array}
      \right.
$$
and
$$b_{e^*}:=\left\{
        \begin{array}{ll}
          e^* & \mathrm{if~}r(e)\in(E^0\setminus H)\setminus(B_H\setminus S)  \\
          e^*+(e')^* & \mathrm{if~} r(e)\in B_H\setminus S \\
          0 & \mathrm{if~} r(e)\in H
        \end{array}
      \right.
$$
in $L_R(E/(H,S))$. It is easy to check that the set $\{a_v,b_e,b_{e^*}\}$ is a Leavitt $E$-family in $L_R(E/(H,S))$. So, the universal property of $L_R(E)$ implies that there is an $R$-algebra homomorphism $\phi :L_R(E)\rightarrow L_R(E/(H,S))$ such that $\phi(v)=a_v$, $\phi(e)=b_e$, and $\phi(e^*)=b_{e^*}$. Also, since $\phi$ vanishes on the generators $\{v:v\in H\}\cup\{v^H:v\in S\}$ of $I(H,S)$, we have $I(H,S)\subseteq \ker \phi$.

Now, if $rv\in I(H,S)$ for some $v\in E^0$ and some $r\in R\setminus\{0\}$, then $r\phi(v)=\phi(rv)=0$. But, by the definition, we have
$$\phi(v)=\left\{
            \begin{array}{ll}
              v & \mathrm{if~}v\in(E^0\setminus H)\setminus(B_H\setminus S) \\
              v+v' & \mathrm{if~} v\in B_H \\
              0 & \mathrm{if~} v\in H,
            \end{array}
          \right.
$$
and since the vertices of $E/(H,S)$ are linearly independent \cite[Proposition 4.9]{tom1}, we get that $\phi(v)=0$ and $v\in H$. Similarly, if $rv^H\in I$ for some $v\in B_H$ and some $r\in R\setminus\{0\}$, then $r\phi(v^H)=\phi(rv^H)=0$. Hence, by using the fact
$$\phi(v^H)=\left\{
    \begin{array}{ll}
      0 & \mathrm{if~} v\in S \\
      v' & \mathrm{if~} v\in B_H\setminus S
    \end{array}
  \right.
$$
and applying \cite[Proposition 3.4]{tom1}, we conclude that $v\in S$. Therefore, the claim holds.

Now we show that $I(H,S)$ is a basic ideal. Assume that $rx\in I(H,S)$ for $r\in R\setminus\{0\}$ and $x=v-\sum_{i=1}^ne_ie_i^*$ with $s(e_i)=v$. We show that $x\in I(H,S)$. If $v\in H$, then $v,e_ie_i^*\in I(H,S)$ for all $1 \leq i\leq n$ and hence $x\in I$.

If $v\in E^0\setminus H$, by rearranging terms, we may write
$$x=v-\sum_{i=1}^me_ie_i^*-\sum_{i=m+1}^n e_ie_i^*$$
where $r(e_i)\notin H$ for $1\leq i\leq m$ and $r(e_i)\in H$ for $m+1\leq i\leq n$. Note that if $r(e)\in H$ for an edge $e$, then $ee^*=e r(e)e^*\in I(H,S)$. So $\sum_{i=m+1}^n e_ie_i^*\in I(H,S)$ and for showing $x\in I(H,S)$ it suffices to show $v-\sum_{i=1}^me_ie_i^*\in I(H,S)$. To this end, we consider four possibilities for $v$. Also, note that
$$rv-r\sum_{i=1}^m e_ie_i^*=rx+r\sum_{i=m+1}^ne_ie_i^*\in I(H,S).$$

$\underline{\emph{Case I}}$ : $v\in B_H$.\\
Then
$$rv^H=v^H\left(rv-r\sum_{i=1}^m e_ie_i^*\right)\in I(H,S)$$
and so the first part of proof gives that $v\in S$ and $v^H\in I(H,S)$. We claim that $v-\sum_{i=1}^me_ie_i^*=v^H$. To obtain a contradiction we assume that there exists an edge $e$ with $s(e)=v$, $r(e)\in E^0\setminus H$, and $e\neq e_i$ for $1\leq i\leq m$. Then
$$rr(e)=e^*\left(rv-r\sum_{i=1}^m e_ie_i^*\right)e\in I(H,S)$$
and so $r(e)\in H$, contradicting the hypothesis. Therefore, $v-\sum_{i=1}^m e_ie_i^*=v^H\in I(H,S)$.

$\underline{\emph{Case II}}$ : $v$ emits no edges  into $E^0\setminus H$.\\
Then $m=0$ and $rv=rx+r\sum_{i=1}^ne_i e_i^*\in I(H,S)$. So by the first part of proof, we have $v\in H\subseteq I(H,S)$.

$\underline{\emph{Case III}}$ : $v$ emits infinitely many edges into $E^0\setminus H$.\\
Then we may choose an edge $e$ such that $s(e)=v$, $r(e)\in E^0\setminus H$, and $e\neq e_i$ for all $1\leq i\leq n$. Hence
$$rr(e)=e^*\left(rv-r\sum_{i=1}^m e_ie_i^*\right)e\in I(H,S)$$
and so $r(e)\in H$, a contradiction. Therefore, this case does not happen for $v$.

$\underline{\emph{Case IV}}$ : $v$ emits finitely many edges into $H$ and emits finitely many edges into $E^0\setminus H$.\\
Then $v$ is a finite emitter, and by the relation (4) in Definition \ref{defn1.3}, we may write $v=\sum_{i=1}^ke_ie_i^*$, where $s^{-1}(v)=\{e_1,\ldots,e_k\}$. For any index $j\geq m+1$, we have
$$rr(e_j)=e_j^*\left(rv-r\sum_{i=1}^m e_ie_i^*\right)e_j\in I(H,S),$$
and hence, $r(e_j)\in H$ and $e_je_j^*\in I(H,S)$. Therefore,
$$v-\sum_{i=1}^m e_ie_i^*=\sum_{i=1}^ke_ie_i^*-\sum_{i=1}^me_ie_i^*=\sum_{i=m+1}^ke_ie_i^*\in I(H,S).$$
This completes the proof.
\end{proof}

Now we are in the position to prove the main result of this section. For a hereditary subset $X$ of $E^0$, we denote $I(X)$ the ideal of $\lr$ generated by $X$

\begin{theorem}\label{thm3.10}
Let $E$ be a graph and let $R$ be a unital commutative ring. Then
\begin{enumerate}[$(1)$]
\item If $X$ is a hereditary subset of $E^0$, then the ideal $I(X)$ and the Leavitt path algebra $L_R(E_X)$ are Morita equivalent as rings, where $E_X:=(X,s^{-1}(X),r,s)$.
\item If $(H,S)$ is an admissible pair in $E$, then the ideal $I(H,S)$ and the Leavitt path algebra $L_R(E_{(H,S)})$ are Morita equivalent as rings, where $E_{(H,S)}$ is the graph defined in Definition \ref{defn3.6}.
\item If $(H,S)$ is an admissible pair $E$, then the quotient algebra $L_R(E)/I(H,S)$ is canonically isomorphic to the Leavitt path algebra $L_R(E/(H,S))$.
\item The map $(H,S)\mapsto I(H,S)$ is a one-to-one correspondence between the set of admissible pairs in $E$ and the set of graded basic ideals of $L_R(E)$.
\end{enumerate}
\end{theorem}

\begin{proof}

(1) and (2). By using of Lemmas \ref{lem3.7} and \ref{lem3.8}, the proof of these parts are exactly similar to that of \cite[Theorem 7.9(3)]{tom1} and \cite[Theorem 5.7(3)]{tom2}, respectively.

(3). Let $E/(H,S)$ be the quotient graph in Definition \ref{defn3.6}. For each $v,v'\in (E/(H,S))^0$ and $e,e'\in (E/(H,S))^1$ define
$$\hspace{-4mm}\left\{
    \begin{array}{ll}
      a_v:=v & \hspace{4.7mm} \mathrm{if}~v\in(E^0\setminus H)\setminus(B_H\setminus S) \\
      a_v:=v-v^H & \hspace{4.7mm} \mathrm{if}~v\in B_H\setminus S  \\
      a_{v'}:=v^H & \hspace{4.7mm} \mathrm{if}~v\in B_H\setminus S,
    \end{array}
  \right.
$$
$$\left\{
    \begin{array}{ll}
      b_e:=e & \hspace{5.5mm}\mathrm{if}~r(e)\in(E^0\setminus H)\setminus(B_H\setminus S) \\
      b_e:=ea_{r(e)} & \hspace{5.5mm}\mathrm{if}~r(e)\in B_H\setminus S\\
      b_{e'}:=ea_{r(e)'} & \hspace{5.5mm}\mathrm{if}~r(e)\in B_H\setminus S,
    \end{array}
  \right.
$$
and
$$\left\{
    \begin{array}{ll}
      b_{e^*}:=e^* & \hspace{2.5mm}\mathrm{if}~r(e)\in(E^0\setminus H)\setminus(B_H\setminus S)\\
      b_{e^*}:=a_{r(e)}e^* & \hspace{2.5mm}\mathrm{if}~r(e)\in B_H\setminus S\\
      b_{e'^*}:=a_{r(e)'}e^* & \hspace{2.5mm}\mathrm{if}~r(e)\in B_H\setminus S
    \end{array}
  \right.
$$
in $L_R(E)$. If $\pi:L_R(E)\rightarrow L_R(E)/I(H,S)$ is the quotient map, write $A_v:=\pi(a_v)$, $B_e:=\pi(b_e)$, and $B_{e^*}:=\pi(b_{e^*})$ for all $v\in(E/(H,S))^0$ and all $e\in (E/(H,S))^1$. It is straightforward to check that $\{A_v,B_e,B_{e^*}\}$ is a Leavitt $E/(H,S)$-family and thus the universal property gives the homomorphism $\phi:L_R(E/(H,S))\rightarrow L_R(E)/I(H,S)$ with $\phi(v)=A_v$, $\phi(e)=B_e$, and $\phi(e^*)=B_{e^*}$ for $v\in(E/(H,S))^0$ and $e\in(E/(H,S))^1$.

For every $v\in (E/(H,S))^0$, Lemma \ref{lem3.7} yields that $a_v\notin I(H,S)$ and hence we have $\phi(rv)=rA_v\neq 0$ for all $r\in R\setminus\{0\}$ because $I(H,S)$ is a basic ideal (Lemma \ref{lem3.9}). Therefore, the Graded Uniqueness Theorem \cite[Theorem 5.3]{tom1} implies  that $\phi$ is injective.

To see that $\phi$ is surjective, note that $L_R(E)/I(H,S)$ is generated by $\{\pi(v), \pi(e), \pi(e^*):v,r(e)\in E^0\setminus H\}$ and we can recover these elements from $\{A_v,B_e,B_{e^*}\}$:
$$\pi(v)=\left\{
    \begin{array}{ll}
      A_v & \mathrm{if}~v\in(E^0\setminus H)\setminus(B_H\setminus S) \\
      A_v+A_{v'} & \mathrm{if}~v\in B_H\setminus S,
    \end{array}
  \right.
$$
$$
\pi(e)=\left\{
         \begin{array}{ll}
           B_e & \mathrm{if}~r(e)\in(E^0\setminus H)\setminus(B_H\setminus S)\\
           B_e+B_{e'} & \mathrm{if}~r(e)\in(B_H\setminus S),
         \end{array}
       \right.
$$
and
$$\pi(e^*)=\left\{
             \begin{array}{ll}
               B_{e^*} & \mathrm{if}~r(e)\in (E^0\setminus H)\setminus (B_H\setminus S) \\
               B_{e^*}+B_{e'^*} & \mathrm{if}~r(e)\in B_H\setminus S.
             \end{array}
           \right.
$$
Hence $\phi$ is surjective, and therefore, $L_R(E)/I(H,S)\cong L_R(E/(H,S))$.

(4). The injectivity of the map follows immediately from Lemma \ref{lem3.7}. Indeed, if $I(H_1,S_1)=I(H_2,S_2)$, then by Lemma \ref{lem3.7} we have
$$H_1=\{v\in E^0:v\in I(H_1,S_1)=I(H_2,S_2)\}=H_2$$
and
$$S_1=\{v\in B_{H_1}=B_{H_2}:v^H\in I(H_1,S_1)=I(H_2,S_2)\}=S_2.$$

To see that the map is surjective, let $I$ be a graded basic ideal of $L_R(E)$ and set $H:=\{v\in E^0:v\in I\}$ and $S:=\{v\in B_H:v^H\in I\}$. Then $(H,S)$ is an admissible pair in $E$ and $I(H,S)\subseteq I$. For the reverse inclusion, note that $I$ and $I(H,S)$ contain the same $v$'s and $w^H$'s by Lemma \ref{lem3.7}. Hence, since $I$ is a basic ideal, it follows that $rv\notin I$ and $sw^H\notin I$  for all $v\in E^0\setminus H$, $w\in B_H\setminus S$, and $r,s\in R\setminus\{0\}$. Therefore, if $\pi: L_R(E)/I(H,S) \cong L_R(E/(H,S))\rightarrow L_R(E)/I$ is the quotient map, we have $\pi(rv)\neq0$, for every $v\in(E/(H,S))^0$ and $r\in R\setminus\{0\}$. Also $\pi$ is a graded homomorphism because $I$ is a graded ideal. Now the Graded Uniqueness Theorem \cite[Theorem 5.3]{tom1} implies that $\pi$ is injective and hence $I=I(H,S)$.
\end{proof}

\begin{remark} Observe that if $R=K$ is a field, then every ideal of a Leavitt path algebra $L_K(E)$ is basic. Hence, Theorem \ref{thm3.10} is the generalization of \cite[Theorem 5.7]{tom2} for the case that $R$ is a unital commutative ring. Also, this theorem is an extension of \cite[Theorem 7.9]{tom1} which was proved for row-finite graphs. However, our method of proof is similar to those used in \cite{tom2,tom2} which came from \cite{bat1} for the graph $C^*$-algebras.
\end{remark}

Theorem \ref{thm3.10}(3) shows that every quotient of a Leavitt path algebra by a graded basic ideal is in the class of Leavitt path algebras. This result is an important tool to study ideal structure of Leavitt path algebras. Our next result, Proposition \ref{prop3.13}, is similar to \cite[Theorem 6.1]{rui} which shows that every graded basic ideal of a Leavitt path algebra is also a Leavitt path algebra. Before that we state the following definition from \cite{rui}

\begin{definition}[\cite{rui}, Definition 4.1]\label{defn3.12}
Suppose that $E$ is a graph and $(H,S)$ is an admissible pair in $E$. Let\\
$$F_1(H,S):=\left\{\alpha=e_1\ldots e_n\in \mathrm{Path}(E):s(e_n)\notin H\cup S \mathrm{~and~} r(e_n)\in H\right\}$$
and
$$F_2(H,S):=\left\{\alpha\in \mathrm{Path}(E):|\alpha|>1 \mathrm{~and~} r(\alpha)\in S\right\}.$$
Let $\overline{F_i}(H,S):=\{\overline{\alpha}:\alpha\in F_E(H,S)\}$ be another copy of $F_i(H,S)$ for $i=1,2$. We define the graph $_HE_S$ as follows:
\begin{align*}
_HE^0_S&:=H\cup S\cup F_1(H,S)\cup F_2(H,S),\\
_HE^1_S&:=s^{-1}(H)\cup \{e\in E^1:s(e)\in S~\mathrm{and}~r(e)\in H\}\cup \overline{F_1}(H,S)\cup \overline{F_2}(H,S)
\end{align*}
with $s(\overline{\alpha})=\alpha$ and $r(\overline{\alpha})=r(\alpha)$ for $\overline{\alpha}\in \overline{F_1}(H,S)\cup \overline{F_2}(H,S)$, and the source and range as in $E$ for the other edges of $_HE^1_S$.
\end{definition}

\begin{proposition}[See Theorem 6.1 of \cite{rui}]\label{prop3.13}
Let $E$ be a graph and let $(H,S)$ be an admissible pair in $E$. Then $I(H,S)\cong L_R(_HE_S)$ as rings.
\end{proposition}

\begin{proof}
We define a map $\phi:L_R(_HE_S)\rightarrow I(H,S)$ by the following rules:
\begin{align*}
\phi(v)&=\left\{
          \begin{array}{ll}
            v & \mathrm{~if~}v\in H \\
            v^H & \mathrm{~if~} v\in S
          \end{array}
        \right.\\
\phi(\alpha)&=\left\{
                \begin{array}{ll}
                  \alpha\alpha^* & \mathrm{~if~} \alpha\in F_1(H,S)\\
                  \alpha r(\alpha)^H\alpha^* & \mathrm{~if~} \alpha\in F_2(H,S)
                \end{array}
              \right.\\
\phi(e)&=e ~\mathrm{if}~ e\in s^{-1}(H)\cup \{e\in E^1:s(e)\in S~\mathrm{and}~r(e)\in H\},\\
\phi(\overline{\alpha})&=\left\{
                            \begin{array}{ll}
                              \alpha & ~\mathrm{~if~} \overline{\alpha}\in \overline{F_1}(H,S)\\
                              \alpha r(\alpha)^H & ~~\mathrm{~if~}\overline{\alpha}\in \overline{F_2}(H,S),
                            \end{array}
                          \right.
\end{align*}
and for each $f\in~\hspace{-2mm} _HE^1_S$ we set $\phi(f^*)=\phi(f)^*$. Similar to the proof of \cite[Theorem 5.1]{rui} one may check that the set $\{\phi(v),\phi(e),\phi(e^*):v\in ~\hspace{-2mm}_HE^0_S,e\in ~\hspace{-2mm}_HE^1_S\}$ is a Leavitt $_HE_S$-family in $I(H,S)$ and hence the universal property implies that there exists such homomorphism. We will show that $\phi$ is an isomorphism. The proof of surjectivity of $\phi$ is identical to the last part of the proof of \cite[Theorem 5.1]{rui}, so we omit it.

It remains to show that $\phi$ is injective. Define a degree map $d$ by
\begin{align*}
&\left\{
  \begin{array}{ll}
           d(v):=0 & \mathrm{~if~} v\in H,\\
           d(v^H):=0 & \mathrm{~if~} v\in S,
\end{array}
\right.\\
&\left\{
  \begin{array}{ll}
           d(e):=1 & \mathrm{~if~} r(e)\in H \\
           d(er(e)^H):=1 & \mathrm{~if~}r(e)\in S\\
           d(e(r(e)-r(e)^H)):=0 & \mathrm{~if~}r(e)\in S\\
           d(e):=0 & \mathrm{~if~} r(e)\notin H\cup S,
\end{array}
\right.\\
\hspace{-20mm}\mathrm{and}\\
&\left\{
  \begin{array}{ll}
           d(e^*):=-1 & \mathrm{~if~} r(e)\in H \\
           d(r(e)^He^*):=-1 & \mathrm{~if~}r(e)\in S\\
           d((r(e)-r(e)^H)e^*):=0 & \mathrm{~if~}r(e)\in S\\
           d(e^*):=0 & \mathrm{~if~} r(e)\notin H\cup S.
\end{array}
\right.\\
\end{align*}
By the span in Lemma \ref{lem3.3}, $d$ induces a $\Z$-grading on $I(H,S)$, and hence $I(H,S)$ is a graded algebra. Note that if $s(e)=v\in S$ and $r(e)\notin H$ then $ee^*e=(v-v^H)e$. Thus if $\overline{\alpha}=\overline{e_1\ldots e_n}\in \overline{F_2}(H,S)$, we have

\begin{align*}
d\left(\phi(\overline{\alpha})\right)&=d\left(e_1e_2\ldots e_nr(e_n)^H\right)\\
&=d\left(e_1(e_2e_2^*)e_2\ldots(e_ne_n^*)e_nr(e_n)^H\right)\\
&=d(e_1(e_2e_2^*))+\ldots+d(e_{n-1}(e_ne_n^*))+d(e_n r(e_n)^H)\\
&=0+\ldots+0+1=1.
\end{align*}
Also, if $\overline{\alpha}=\overline{e_1\ldots e_n}\in \overline{F_1}(H,S)$, then
\begin{align*}
d(\phi(\overline{\alpha}))&=d(e_1\ldots e_n)\\
&=d(e_1)+\ldots+d(e_{n-1})+d(e_n)\\
&=0+\ldots+0+1=1.
\end{align*}
Hence for every $x\in ~\hspace{-2mm}_HE^0_S$ and $f\in ~\hspace{-2mm}_HE^1_S$, $\phi(x)$, $\phi(f)$, and $\phi(f^*)$ are homogeneous of degree $0$, $1$, and $-1$, respectively. This implies that $\phi$ is a graded homomorphism. Furthermore, we have $\phi(rx)\neq 0$ for all $x\in ~\hspace{-2mm} _HE^0_S$ and $r\in R\setminus\{0\}$ because all elements $r\alpha\alpha^*$ and $r\alpha r(\alpha)^H\alpha^*$ are nonzero. Therefore, $\phi$ satisfies the conditions of the Graded Uniqueness Theorem, and hence, $\phi$ is injective.
\end{proof}

\begin{corollary}\label{cor3.14}
Let $E$ be a graph and let $R$ be a unital commutative ring. If $(H,S)$ is an admissible pair in $E$, then $I(H,S)$ is a ring with a set of local units.
\end{corollary}

In the next theorem we show that if a graph $E$ satisfying Condition (K), then every basic ideal of $L_R(E)$ is graded. Recall that a graph $E$ is said to satisfy Condition (L) if every closed simple path has an exit, and is said to satisfy Condition (K) if any vertex in $E^0$ is either the base of no closed path or the base of at least two distinct closed simple paths. The following lemmas are well-known in literature.

\begin{lemma}[\cite{tom2}, Proposition 6.12]\label{lem3.15}
A graph $E$ satisfies Condition (K) if and only if for every admissible pair $(H,S)$ in $E$ the quotient graph $E/(H,S)$ satisfies Condition (L).
\end{lemma}

\begin{lemma}[\cite{tom1}, Lemma 7.14]
If $E$ is a simple closed path of length $n$, then $L_R(E)\cong M_n(R[x,x^{-1}])$.
\end{lemma}

The following lemma has been proved in row-finite case in \cite[Proposition 7.16]{tom1}. By using Lemma \ref{lem3.8} and Corollary \ref{cor3.14}, in the non-row-finite case the proof is identical.

\begin{lemma}\label{lem3.17}
Let $R$ be a commutative ring, and let $E$ be a graph containing a closed simple path without exits. Then $L_R(E)$ contains an ideal that is basic but not graded.
\end{lemma}

\begin{theorem}\label{thm3.18}
Let $E$ be a graph and let $R$ be a unital commutative ring. Then $E$ satisfies Condition (K) if and only if every basic ideal of $L_R(E)$ is graded.
\end{theorem}

\begin{proof}
Suppose that $E$ satisfies Condition (K) and $I$ is a basic ideal of $L_R(E)$. We show that $I$ is graded. If $H:=I\cap E^0$ and $S=\{v:v^H\in I\}$, then $(H,S)$ is an admissible pair in $E$ and we have $I(H,S)\subseteq I$. Let $\phi:L_R(E/(H,S))\cong L_R(E)/I(H,S)\rightarrow L_R(E)/I$ be the canonical surjection. Since $I$ is basic, we have $\phi(rv)\neq 0$ for all $v\in (E/(H,S))^0$ and $r\in R\setminus\{0\}$. Also, since $E$ satisfies Condition (K), Lemma \ref{lem3.15} implies that the quotient graph $E/(H,S)$ satisfies Condition (L). Therefore, the Cuntz-Krieger Uniqueness Theorem (\cite[Theorem 6.5]{tom1}) implies that $\phi$ is injective and so $I=I(H,S)$. Hence $I$ is graded by Lemma \ref{lem3.3}.

Conversely, suppose that $E$ does not satisfy Condition (K). By Lemma \ref{lem3.15}, there exists an admissible pair $(H,S)$ in $E$ such that the quotient graph $E/(H,S)$ does not satisfy Condition (L). So there exists a closed simple path without exits in $E/(H,S)$. Thus Lemma \ref{lem3.17} gives that the algebra $L_R(E/(H,S))\cong L_R(E)/I(H,S)$ contains a basic ideal $I$ which is not graded. If $\pi:L_R(E)\rightarrow L_R(E)/I(H,S)$ be the quotient map, then $\pi$ is a graded homomorphism. Hence, $\pi^{-1}(I)$ is a basic ideal of $L_R(E)$ which is not graded. This completes the proof.
\end{proof}

\begin{corollary}\label{cor3.19}
Let $E$ be a graph satisfying Condition (K) and let $R$ be a unital commutative ring. Then the map $(H,S)\mapsto I(H,S)$ is a one-to-one correspondence between the set of admissible pairs in $E$ and the set of basic ideals of $L_R(E)$.
\end{corollary}

\section{Prime ideals of Leavitt path algebras}

Having determined the graded basic ideals of a Leavitt path algebra $L_R(E)$, in this section we characterize the prime graded basic ideal of $L_R(E)$. If $E$ satisfies Condition (K) and $R$ is a field, every ideal of $L_R(E)$ is graded and basic. Therefore in this case, all prime ideals of $L_R(E)$ are determined. Our results are generalizations of the results of \cite[Section 2]{ara3} which were for a Levitt path algebra $L_K(E)$, where $E$ is a row-finite graph and $K$ is a field. We show that there is a relation between the prime graded basic ideals of $L_R(E)$ and special subsets of $E^0$ which are called maximal tails in $E$.

For $v,w\in E^0$, we denote $v\geq w$ if there exists a path from $v$ to $w$. Also, for a vertex $v\in E^0$ the tree of $v$ is the set $T(v):=\{w\in E^0:v\geq w\}$. Note that $T(v)$ is a hereditary  subset of $E^0$, and if $I$ is an ideal of $\lr$ containing $v$, then $\overline{T(v)}\subseteq I$. Now we recall the definition of maximal tails in $E$.

\begin{definition}[\cite{bat1}]
Let $E$ be a graph. A nonempty subset $M\subseteq E^0$ is called a \emph{maximal tail} if it satisfies the following properties:
\begin{enumerate}[(MT1)]
\item if $v\in E^0,w\in M,$ and $v\geq w$, then $v\in M$,
\item if $v\in M$ with $0<|s^{-1}(v)|<\infty$, then there exists $e\in E^1$ such that $s(e)=v$ and $r(e)\in M$, and
\item for every $v,w\in M$ there exists $y\in M$ such that $v\geq y$ and $w\geq y$.
\end{enumerate}
\vspace{1mm}
The collection of all maximal tails in $E$ is denoted by $M(E)$. Note that a subset $X\subseteq E^0$ satisfies Conditions (MT1) and (MT2) if and only if $E^0\setminus X$ is hereditary and saturated. Also, if $M$ is a maximal tail in $E$, Condition (MT3) implies that $T(v)\cap T(w)\cap M\neq\emptyset$ for any pair $v,w\in M$.
\end{definition}
We need the following lemma to prove Lemma \ref{lem4.2}.
\begin{lemma}\label{lem4.1}
Let $E$ be a graph, and let $(H_1,S_1)$ and $(H_2,S_2)$ be two admissible pairs in $E$. Then
$$I(H_1,S_1)\cap I(H_2,S_2)=I\left(H_1\cap H_2,~(H_1\cup S_1)\cap(H_2\cup S_2)\cap B_{H_1\cap H_2}\right).$$
\end{lemma}

\begin{proof}
For simplicity, we set $I_1:=I(H_1,S_1)$ and $I_2:=I(H_2,S_2)$. Since the intersection of graded basic ideals is also graded basic, Theorem \ref{thm3.10}(4) implies that $I_1\cap I_2=I(K,T)$ for an admissible pair $(K,T)$. We show that $K=H_1\cap H_2$ and $T=(H_1\cup S_1)\cap(H_2\cup S_2)\cap B_{H_1\cap H_2}$.

By applying Lemma \ref{lem3.7} twice, we get that
$$K=I(K,T)\cap E^0=(I_1\cap E^0)\cap (I_2\cap E^0)=H_1\cap H_2.$$
It remains to show $T=(H_1\cup S_1)\cap(H_2\cup S_2)\cap B_K$. Note that Lemma \ref{lem3.7} implies that $T=\{v\in B_K:v^K\in I_1\cap I_2\}$. Let $v\in B_K$; we show that $v\in (H_1\cup S_1)\cap(H_2\cup S_2)$ if and only if $v\in T$. This is equivalent to $v\in(H_1\cup S_1)\cap(H_2\cup S_2)$ if and only if $v^K\in I_1\cap I_2$.

Suppose that $v\in (H_1\cup S_1)\cap(H_2\cup S_2)$ and fix $i\in\{1,2\}$. If $v\in H_i$, then
$$v^K=v-\sum_{s(e)=v,r(e)\notin K}ee^*=v-\sum_{s(e)=v,r(e)\notin K} vee^*\in I_i.$$
If $v\in S_i$, then $v^{H_i}\in I_i$ and so

\begin{align*}
v^K&=v-\sum_{s(e)=v,r(e)\notin K}ee^*\\
&=v-\sum_{s(e)=v,r(e)\notin H_i}ee^*-\sum_{s(e)=v,r(e)\in H_i\setminus K}ee^*\\
&=v^{H_i}-\sum_{s(e)=v,r(e)\in H_i\setminus K}ee^*\in I_i.
\end{align*}
(Note that if $r(e)\in H_i$, then $ee^*=er(e)e^*\in I_i$.) So, in all cases we have $v^K\in I_i$. Consequently, $v^K\in I_1\cap I_2$.

For the reverse, suppose that $v^K\in I_1\cap I_2$ and fix $i\in \{1,2\}$. If for each of the finitely many edges $e$ with $s(e)=v$ and $r(e)\notin K$ we have $r(e)\in H_i$, then $v=v^K+\sum_{s(e)=v,r(e)\notin K}ee^*\in I_i$ and so $v\in H_i$. If $r(e)\notin H_i$ for some such $e$, then
\begin{align*}
v^{H_i}&=v-\sum_{s(e)=v,r(e)\notin H_i}ee^*\\
&=v-\sum_{s(e)=v,r(e)\notin K}ee^*+\sum_{s(e)=v,r(e)\in H_i\setminus K}ee^*\\
&=v^{K}+\sum_{s(e)=v,r(e)\in H_i\setminus K}ee^*\in I_i
\end{align*}
and so $v\in S_i$ by Lemma \ref{lem3.7}. Therefore, in each case we have $v\in H_i\cup S_i$ and the proof is completed.
\end{proof}

Recall that an ideal (graded ideal) $I$ of a ring (graded ring) $R$ is said to be \emph{prime} (\emph{graded prime}) if for every pair of ideals (graded ideals) $J_1,J_2$ of $R$ with $J_1J_2\subseteq I$, we have either $J_1\subseteq I$ or $J_2\subseteq I$. It follows from \cite[Propositin II.1.4]{nas} that for a graded algebra, a graded ideal is prime if and only if it is graded prime. We use this fact in the proof of Proposition \ref{prop4.4}. Also, a ring $R$ is said to be prime if the zero ideal of $R$ is prime.

\begin{lemma}\label{lem4.2}
Let $E$ be a graph and let $R$ be a unital commutative ring. If $I$ is a prime ideal of $L_R(E)$ and $M:=E^0\setminus(I\cap E^0)$, then $M$ is a maximal tail in $E$.
\end{lemma}

\begin{proof}
Since $I\cap E^0$ is hereditary and saturated, $M$ satisfies  Conditions (MT1) and (MT2).

We show that $M$ also satisfies Condition (MT3). We assume on the contrary that there exist $v,w\in M$ such that there is no $y\in M$ as (MT3). Then $T(v)\cap T(w)\cap M=\emptyset$ and so $T(v)\cap T(w)\subseteq I\cap E^0$.

Furthermore, the set $(I\cap E^0)\cup(E^0\setminus T(v))$ is saturated. Indeed, let $x\in E^0$ be such that $0<|s^{-1}(x)|<\infty$ and $r(e)\in (I\cap E^0)\cup(E^0\setminus T(v))$ for each edge $e$ with $s(e)=x$. If for all such edges $r(e)\in I\cap E^0$, then we have $x\in I\cap E^0$ because $I\cap E^0$ is saturated. If there is at least one $e$ with $s(e)\notin T(v)$, then we have $x\in E^0\setminus T(v)$ by the hereditary property of $T(v)$. Hence in all cases we have $x\in (I\cap E^0)\cup(E^0\setminus T(v))$ and so, $(I\cap E^0)\cup(E^0\setminus T(v))$ is saturated. Thus $T(w)\subseteq (I\cap E^0)\cup(E^0\setminus T(v))$ yields that $\overline{T(w)}\subseteq (I\cap E^0)\cup(E^0\setminus T(v))$. Similarly, we get that $\overline{T(v)}\subseteq (I\cap E^0)\cup(E^0\setminus T(w))$, and the hypothesis $T(v)\cap T(w)\subseteq I\cap E^0$ implies that $\overline{T(v)}\cap \overline{T(w)}\subseteq I\cap E^0$. So by Lemma \ref{lem4.1}, we have
$$I(\overline{T(v)},\emptyset)\cap I(\overline{T(w)},\emptyset)=I(\overline{T(v)}\cap \overline{T(w)},\emptyset)\subseteq I.$$
But the primeness of $I$ implies that either $I(\overline{T(v)},\emptyset)\subseteq I$ or $I(\overline{T(w)},\emptyset)\subseteq I$, which is impossible because neither the ideal of $L_R(E)$ generated by $v$ nor generated by $w$ is contained in $I$. Thus there must exist $y\in M$ such that $v\geq y$ and $w\geq y$, and consequently, $M$ satisfies Condition (MT3).
\end{proof}

\begin{lemma}\label{lem4.3}
Let $I$ be a nonzero graded ideal of $L_R(E)$. Then there exists $rv\in I$ for some $v\in E^0$ and some $r\in R\setminus\{0\}$.
\end{lemma}

\begin{proof}
Since $I$ is a graded ideal, $I$ is generated by $I\cap L_R(E)_0$ (see \cite[Lemma 5.1]{tom1}). If $x$ is a nonzero element in $I\cap L_R(E)_0$, then \cite[Lemma 5.2]{tom1} implies that there exist $\alpha,\beta\in \mathrm{Path}(E)$ such that $\alpha^* x\beta=rv$ for some $v\in E^0$ and some $r\in R\setminus\{0\}$. Therefore, $rv\in I$ as desired.
\end{proof}

In the next proposition, we give two necessary and sufficient conditions for the primeness of Leavitt path algebras. Furthermore, this results will be a tool for the next results of this section to characterize prime graded basic ideals of Leavitt path algebras.

\begin{proposition}\label{prop4.4}
Let $E$ be a graph and let $R$ be a unital commutative ring. Then the following are equivalent.
\begin{enumerate}[$(1)$]
\item $L_R(E)$ is a prime ring.
\item $R$ is an ID (integral domain) and $E$ satisfies Condition (MT3).
\item $R$ is an ID and $I\cap J\cap E^0\neq \emptyset$ for all graded basic ideals $I$ and $J$ of $\lr$.
\end{enumerate}
\end{proposition}

\begin{proof}
(1) $\Rightarrow$ (2). Suppose that $L_R(E)$ is a prime ring. Then $E^0$ is a maximal tail by Lemma \ref{lem4.2} and so $E$ satisfies Condition (MT3). If there exist nonzero elements $r,s\in R$ with $rs=0$, then $rL_R(E)$ and $sL_R(E)$ are nonzero ideals of $L_R(E)$ with $(rL_R(E))(sL_R(E)) = rsL_R(E)= (0)$, contradicting the primeness of $L_R(E)$.

(2) $\Rightarrow$ (3). If $E$ satisfies Condition (MT3) and $I,J$ are two graded basic ideals of $\lr$, then by Theorem \ref{thm3.10}(4) there exist vertices $v\in I$ and $w\in J$. Hence, Condition $(MT3)$ yields that there is $y\in E^0$ such that $v,w\geq y$ and so $y\in I\cap J$ as desired.

(3) $\Rightarrow$ (1). Suppose that $R$ is an ID and $I\cap J\cap E^0\neq \emptyset$ for every graded basic ideals $I$ and $J$ of $\lr$. We show that the zero ideal of $L_R(E)$ is prime. Since the zero ideal is graded, it suffices to show $IJ\neq (0)$ for any pair of nonzero graded ideals $I,J$ of $L_R(E)$. Let $I$ and $J$ be two nonzero graded ideals of $L_R(E)$. Then, by Lemma \ref{lem4.3}, there exist elements $rv\in I$ and $sw\in J$, where $v,w\in E^0$ and $r,s\in R\setminus\{0\}$. If we let $I':=I(\overline{T(v)},\emptyset)$ and $J':=I(\overline{T(w)},\emptyset)$, then $rI'\subseteq I$ and $sJ'\subseteq J$. Since $I'$ and $J'$ are two graded basic ideals, there is $y\in I'\cap J'\cap E^0$ and so $rsy\in IJ$. On the other hand, because $R$ is an ID, $rs$ is a nonzero element of $R$ and hence $rsy\neq 0$. Therefore, $IJ\neq (0)$ and consequently the zero ideal is prime.
\end{proof}

\begin{corollary}\label{cor4.5}
Let $E$ be a graph and let $R$ be a unital commutative ring. If $H$ is a saturated hereditary subset of $E^0$, then the ideal $I(H,B_H)$ is prime if and only if $R$ is an ID and $M:=E^0\setminus H$ is a maximal tail.
\end{corollary}

\begin{proof}
Recall that $E/(H,B_H)=(E^0\setminus H,r^{-1}(E^0\setminus H),r,s)$. Thus the statement follows from Proposition \ref{prop4.4} and the fact $L_R(E)/I(H,B_H)\cong L_R(E/(H,B_H))$.
\end{proof}

Following \cite{bat1}, for a non-empty subset $X$ of $E^0$ we denote
$$\Omega(X):=\{w\in E^0\setminus X: w\ngeq v~\mathrm{for~every~}v\in X\}.$$
In the next theorem, we use this notation to characterize prime graded basic ideals of a Leavitt path algebra. We first prove two lemmas.

\begin{lemma}\label{lem4.6}
If $X\subseteq E^0$ satisfies Condition (MT1), then $\Omega(X)=E^0\setminus M$.
\end{lemma}

\begin{proof}
By definition we have $\Omega(M)\subseteq E^0\setminus M$. For the reverse inclusion, let $v\in E^0\setminus M$. If there is a vertex $w\in M$ with $v\geq w$, then we have $v\in M$ by Condition (MT1), which is a contradiction. Therefore, $E^0\setminus M\subseteq \Omega(M)$.
\end{proof}

\begin{corollary}
If $X\subseteq E^0$ satisfies Conditions (MT1) and (MT2), then $\Omega(X)$ is hereditary and saturated. In particular, if $M$ is a maximal tail in $E$, then $\Omega(M)=E^0\setminus M$ is hereditary and saturated.
\end{corollary}

\begin{lemma}\label{lem4.8}
Let $E$ be a graph and let $R$ be a unital commutative ring. If $H$ is a saturated hereditary subset of $E^0$ and $v\in B_H$, then the ideal $I(H,B_H\setminus\{v\})$ of $L_R(E)$ is prime if and only if $R$ is an ID and $H=\Omega(v)$.
\end{lemma}

\begin{proof}
First, note that the hereditary property of $H$ implies $H\subseteq\Omega(v)$. On the other hand, $I(H,B_H\setminus\{v\})$ is prime if and only if the zero ideal of the corresponding quotient algebra $L_R(E/(H,B_H\setminus\{v\}))$ is prime. Since the graph $E/(H,B_H\setminus\{v\})$ contains the sink $v'$, Proposition \ref{prop4.4} and Condition (MT3) imply that $I(H,B_H\setminus\{v\})$ is prime if and only if $R$ is an ID and for every vertex $w\in E^0\setminus H$, there exists a path from $w$ to $v$. This is equivalent to $R$ is an ID and $\Omega(v)\subseteq H$.
\end{proof}

\begin{definition}[\cite{bat1}]
Suppose that $v\in E^0$ is an infinite emitter. Then $v$ is called a \emph{breaking vertex} if $v\in B_{\Omega(v)}$. We denote by $BV(E)$ the set of breaking vertices of $E$. That is,
$$BV(E):=\{v\in E^0:|s^{-1}(v)|=\infty~\mathrm{and}~0<|s^{-1}(v)\setminus r^{-1}(\Omega(v))|<\infty\}.$$
Note that if a vertex $v$ is an infinite emitter, then $\Omega(v)$ is automatically hereditary and saturated. Also, if $E$ is a row-finite graph, then $BV(E)=\emptyset$.
\end{definition}

Now we prove the main result of this section. This theorem is the generalization of \cite[Proposition 5.6]{ara2} for the case that $R$ is a unital commutative ring and $E$ is an arbitrary graph.

\begin{theorem}\label{thm4.8}
Let $E$ be a graph and let $R$ be a unital commutative ring. If $R$ is an ID, then the set of prime graded basic ideals of $L_R(E)$ is
$$\left\{I(\Omega(M),B_{\Omega(M)}),I(\Omega(v),B_{\Omega(v)}\setminus\{v\}):M\in M(E), v\in BV(E)\right\}.$$
\end{theorem}

\begin{proof}
By Theorem \ref{thm3.10}(4), every graded basic ideal of $L_R(E)$ is of the form $I(H,S)$ for an admissible pair $(H,S)$ in $E$. Assume first that $B_H\setminus S$ contains two or more vertices. Then the quotient graph $E/(H,S)$ contains at least two sinks, and hence it does not satisfy Condition (MT3). Thus Proposition \ref{prop4.4} gives that the zero ideal of $L_R(E/(H,S))\cong L_R(E)/I(H,S)$ is not prime and so, $I(H,S)$ is not a prime ideal of $L_R(E)$.

Therefore, the prime graded basic ideals of $L_R(E)$ must have the form either $I(H,B_H)$ or $I(H,B_H\setminus\{v\})$ for some $v\in B_H$. If $I(H,B_H)$ is a prime ideal, then $M:=E^0\setminus H\in M(E)$ by Corollary \ref{cor4.5} and $\Omega(M)=H$ by Lemma \ref{lem4.6}. If $I(H,B_H\setminus\{v\})$ is prime, then Lemma \ref{lem4.8} gives that $H=\Omega(v)$.
\end{proof}

The following corollaries are direct consequences of Theorem \ref{thm4.8}. Recall from Theorem \ref{thm3.18} that if $E$ satisfies Condition (K), then every basic ideals of $L_R(E)$ is graded. Also, if $R=K$ is a field, then every ideal of a Leavitt path algebra $L_K(E)$ is basic.

\begin{corollary}
Let $E$ be a graph satisfying Condition (K) and let $R$ be a unital commutative ring. Then Theorem \ref{thm4.8} gives a complete description of the prime basic ideals of $L_R(E)$.
\end{corollary}

\begin{corollary}\label{cor4.13}
Let $E$ be a graph satisfying Condition (K) and let $K$ be a field. Then Theorem \ref{thm4.8} gives a complete description of the prime ideals of $L_K(E)$. In particular, if $E$ has no proper maximal tails and no breaking vertices, then $L_K(E)$ is simple.
\end{corollary}

\begin{proof}
The second statement follows from the first one because in this case $L_K(E)$ has no nonzero prime ideals.
\end{proof}

\section{Minimal left ideals generated by a vertex}

In this section, we give two necessary and sufficient conditions for the minimality of a left ideal and a right ideal generated by a vertex which will appear in the sequel. This is the generalization of the results in \cite[Section 2]{ara1} for a non-row finite setting. We will need this result to characterize primitive Leavitt algebras in Section 6.

\begin{definition}
A vertex $v\in E^0$ is called a \emph{bifurcation} if $v$ emits at least two edges. We say that a path $\alpha=e_1\ldots e_n$ contains no bifurcations if the set $\{s(e_i)\}_{i=1}^n$ contains no bifurcations. A vertex $v\in E^0$ is called a \emph{line point} if any vertex $w\in T(v)$ is neither a bifurcation nor the base of a closed path. The set of all line points in $E$ is denoted by $P_l(E)$.
\end{definition}

\begin{lemma}\label{lem5.1}
Let $v,w\in E^0$ and let $\alpha$ be a path from $v$ to $w$. If $\alpha$ contains no bifurcations, then $\lr v\cong \lr w$ as left $L_R(E)$-module.
\end{lemma}

\begin{proof}
Note that since $\alpha$ has no exits, we have $\alpha\alpha^*=v$ by the relation (4) in Definition \ref{defn1.4}. Now define the epimorphism $\phi:\lr v\rightarrow \lr w$ by $xv\mapsto xv\alpha$. $\phi$ is also injective because if $\phi(xv)= xv\alpha=0$, then we get $xv= xv\alpha\alpha^*=0$. Therefore, $\phi$ is an isomorphism as desired.
\end{proof}

\begin{lemma}\label{lem5.2}
Suppose that $v\in E^0$ and $e_1\ldots e_n\in s^{-1}(v)$. Then $\lr v=\bigoplus_{i=1}^n\lr e_i e_i^*\bigoplus\lr(v-\sum_{i=1}^n e_ie_i^*)$.
\end{lemma}

\begin{proof}
The statement follows from the decomposition $v=(v-\sum e_ie_i^*)+\sum e_ie_i^*$ and the fact that the elements $e_ie_i^*$ and $v-\sum e_ie_i^*$ are pairwise orthogonal.
\end{proof}

\begin{remark}\label{rem5.3}
Lemma \ref{lem5.2} yields that if $v$ is a bifurcation, then the left ideal $\lr v$ is not minimal. Also, $\lr v$ is not a minimal left ideal whenever $T(v)$ contains some bifurcations. Indeed, if $w$ is the first bifurcation in $T(v)$ from $v$, then the left ideal $\lr w$ is not minimal and $\lr v\cong \lr w$ by Lemma \ref{lem5.1}.
\end{remark}

In the next proposition, we show that if a vertex $v\in E^0$ is the base of a closed path, then the left ideal $\lr v$ is not minimal. Recall from Theorem \ref{thm3.10}(1) that if $X$ is a hereditary subset of $E^0$, then $I(X)$ is Morita-equivalent to the Leavitt path algebra $L_R(E_{(X,\emptyset)})$. If we list the elements of $X=\{v_1,v_2,\ldots\}$ and let $u_n:=\sum_{i=1}^n v_i$, then the map $J\mapsto \sum_{n=1}^{|X|} J u_n J$ is a one-to-one correspondence between the lattice of ideals of $I(X)$ and the lattice of ideals of $L_R(E_{(X,\emptyset)})$. Furthermore, we have $I(X)=I(\overline{X},\emptyset)$ and since $I(\overline{X},\emptyset)$ contains a set of local units (by Corollary \ref{cor3.14}), \cite[Lemma 4.14]{tom1} implies that every ideal of $I(X)$ is also an ideal of $\lr$.

\begin{lemma}\label{lem5.4}
Let $E$ be a graph and let $R$ be a unital commutative ring. If there is  some closed path based at $v\in E^0$, then the left ideal $\lr v$ is not minimal.
\end{lemma}

\begin{proof}
Let $\alpha=e_1\ldots e_n$ be a closed simple path based at $v$. Since $T(v)$ contains no bifurcations by Remark \ref{rem5.3}, $X:=\{s(e_i)\}_{i=1}^n$ is a hereditary subset of $E^0$. So Theorem \ref{thm3.10}(1) gives that $I(X)$ is Morita-equivalent to the Leavitt path algebra $L_R(E_{(X,\emptyset)})$. Not that $E_{(X,\emptyset)}$ is a single closed simple path with $n$ vertices and $L_R(E_{(X,\emptyset)})$ is isomorphic to $M_n(R[x,x^{-1}])$ by \cite[Lemma 7.14]{tom1}.

Now consider the ideal $J'=M_n(\langle x+1\rangle)$ in $M_n(R[x,x^{-1}])$ and its corresponding ideal $J$ in $I(X)$. Since $1\notin\langle x+1\rangle$ we have $J'\subsetneq M_n(R[x, x^{-1}])$ and so $J$ is a proper ideal of $I(X)$. Moreover, since $I(X)$ is the ideal of $\lr$ generated by $v$ and $J$ is also an ideal of $\lr$, we have $v\notin J$. Hence, $v\notin Jv$ and $Jv$ is a proper left ideal in $\lr v$. Consequently, $\lr v$ is not a minimal left ideal.
\end{proof}

\begin{proposition}\label{prop5.5}
Let $E$ be a graph and let $R$ be a unital commutative ring. If $v\in E^0$, then the following are equivalent.
\begin{enumerate}[$(1)$]
\item $\lr v$ is a minimal left ideal.
\item $v\lr$ is a minimal right ideal.
\item $v\in P_l(E)$ and $R$ is a field.
\item $v\lr v=Rv$ and $R$ is a field.
\end{enumerate}
\end{proposition}

\begin{proof} We only prove the implications (1) $\Rightarrow$ (3) $\Rightarrow$ (4) $\Rightarrow$ (1). The implications (2) $\Rightarrow$ (3) $\Rightarrow$ (4) $\Rightarrow$ (2) may be proved analogously.

(1) $\Rightarrow$ (3). Assume that $\lr v$ is a minimal left ideal. First note that $R$ is a field because if $I$ is a proper ideal of $R$, then $I\lr v$ is a proper left ideal in $\lr v$. Now we show that $v$ is a line point. Remark \ref{rem5.3} implies that $T(v)$ contains no bifurcations. Also, There is no closed paths based at a vertex in $T(v)$. Indeed, if there is a closed path based at a vertex $w\in T(v)$, then $\lr w$ is not minimal by Lemma \ref{lem5.4}. Since $T(v)$ has no bifurcations, we have $\lr v\cong\lr w$ and so $\lr v$ is not a minimal left ideal, a contradiction. Therefore, $v$ is a line point.

(3) $\Rightarrow$ (4). Assume that $v\in P_l(E)$. Then for every $w\in T(v)$ there is a unique path $\alpha$ from $v$ to $w$, and so $\alpha\alpha^*=s(\alpha)=v$. On the other hand, from Equation \ref{2.1} we have
$$v\lr v=\mathrm{span}_R\left\{\alpha\beta^*:r(\alpha)=r(\beta)~\mathrm{and}~s(\alpha)=s(\beta)=v\right\}.$$
For any element $\alpha\beta^*$ in the above set, $\alpha$ and $\beta$ are paths from $v$ with same ranges. Thus we have $\alpha=\beta$ and $\alpha\beta^*=\alpha\alpha^*=v$. Consequently, $v\lr v=Rv$.

(4) $\Rightarrow$ (1). If $R$ is a field and $v\lr v=Rv$, then the left ideal $\lr v$ is minimal because for every $a\in\lr$ with $av\neq 0$, we have $\lr av=\lr v$. For this, take a nonzero element $avxav\in av\lr av$ by the fact $av\lr av\neq 0$. Then we have $0\neq vxav\in v\lr v=Rv$ and so there is $r\in R$ such that $rvxav=v$. Hence, $v\in\lr av$ and $\lr av=\lr v$.
\end{proof}

\section{Primitive ideals of Leavitt path algebras}

Finally in this section we determine primitive Leavitt path algebras. This result is a generalization of \cite[Theorem 4.6]{ara3} which was proved for the Leavitt path algebras of row-finite graphs with coefficients in a field. We then use this result to characterize primitive graded ideals of Leavitt path algebras. Note that every primitive ideal of an algebra is also prime, and the reverse implication holds for the graph $C^*$-algebras. But the notions of primeness and primitivity do not coincide for the class of Leavitt path algebras. However, we will show in Corollary \ref{cor6.10} that if $R$ is a field and $E$ satisfies Condition (K), then an ideal of $L_R(E)$ is prime if and only if it is primitive.

At the first, in Proposition \ref{prop6.2}, we will give algebraic characterizations of Condition (L) and Conditions (L) plus (MT3). To prove it we need the following lemma.

\begin{lemma}\label{lem6.1}
Let $E$ be a graph satisfying Condition(L) and let $R$ be a unital commutative ring. If $I$ is a nonzero ideal of $L_R(E)$, then there is $rv\in I$ for some $v\in E^0$ and some $r\in R\setminus\{0\}$. In particular, if $I$ is a nonzero basic ideal of $L_R(E)$, then $I\cap E^0\neq \emptyset$.
\end{lemma}

\begin{proof}
Suppose that $I$ is a nonzero ideal of $L_R(E)$ such that $rv\notin I$ for all $v\in E^0$ and $r\in R\setminus\{0\}$. If $\phi: L_R(E)\rightarrow L_R(E)/I$ is the quotient map, then we have $\phi(rv)\notin 0$ for all $v\in E^0$ and $r\in R\setminus\{0\}$. Thus the Cuntz-Krieger Uniqueness Theorem (\cite[Theorem 6.5]{tom1}) implies that $\phi$ is injective and $I=(0)$, a contradiction.
\end{proof}

\begin{proposition}\label{prop6.2}
Let $E$ be a graph and let $R$ be a unital commutative ring.
\begin{enumerate}[$(1)$]
\item $E$ satisfies Condition (L) if and only if $I\cap E^0\neq\emptyset$ for every nonzero basic ideal $I$ of $L_R(E)$.
\item $E$ satisfies Conditions (L) and (MT3) if and only if $I\cap J\cap E^0\neq\emptyset$ for every nonzero basic ideals $I$ and $J$ of $L_R(E)$.
\item $R$ is a field and $E$ satisfies Conditions (L) and (MT3) if and only if $I\cap J\cap E^0\neq\emptyset$ for every nonzero ideals $I$ and $J$ of $L_R(E)$.
\end{enumerate}
\end{proposition}

\begin{proof}
(1). If $E$ satisfies Condition (L) and $I$ is a nonzero basic ideal of $L_R(E)$, then Lemma \ref{lem6.1} gives $I\cap E^0\neq\emptyset$. Conversely, assume that $I\cap E^0\neq\emptyset$ for every nonzero basic ideal $I$ of $L_R(E)$. If $E$ does not satisfy Condition (L), then there is a closed simple path $\alpha=e_1...e_n$ in $E$ with no exits. So $X:=\{s(e_i)\}_{i=1}^n$ is a hereditary subset of $E^0$. If we let $H:=\overline{X}$, a similar argument of the proof of \cite[Lemma 7.16]{tom1} implies that there exits an ideal $J$ of $I(H,\emptyset)$ such that $rv\notin J$ for every $v\in H$. In particular, $J$ is a basic ideal of $I(H,\emptyset)$, and so Corollary \ref{cor3.14} and \cite[Lemma 4.14]{tom1} imply that $J$ is a basic ideal of $L_R(E)$. Since $J\cap E^0=\emptyset$, this contradicts the hypothesis.

(2). Suppose that $E$ satisfies Conditions (L) and (MT3) and take two nonzero basic ideals $I$ and $J$ of $L_R(E)$. Then Lemma \ref{lem6.1} implies that there exists $v\in I$ and $w\in J$ for some $v,w\in E^0$. Now by Condition (MT3), there is a vertex $y$ such that $v,w\geq y$, and so $y\in I\cap J$.

Conversely, assume that $I\cap J\cap E^0\neq \emptyset$ for every basic ideals $I,J$ of $L_R(E)$. Then Part (1) yields that $E$ satisfies Condition (L) by taking $J=\lr$. Now we show that $E$ satisfies Condition (MT3). Let $v,w\in E^0$ and to obtain a contradiction, suppose that there is no $y\in E^0$ as (MT3). If $T(v)$ and $T(w)$ are the trees of $v$ and $w$, respectively, then we have $T(v)\cap T(w)=\emptyset$. Moreover, similar to the proof of Lemma \ref{lem4.2} we have that $E\setminus T(w)$ is a saturated subset of $E^0$. Hence, $T(v)\subseteq E^0\setminus T(w)$ gives that $\overline{T(v)}\subset E^0\setminus T(w)$. Similarly, $\overline{T(w)}\subset E^0\setminus T(v)$ and so $\overline{T(v)}\cap \overline{T(w)}=\emptyset$. Therefore, by Lemma \ref{lem4.1}, we have
$$I(\overline{T(v)},\emptyset)\cap I(\overline{T(w)},\emptyset)=I(\overline{T(v)}\cap \overline{T(w)}, \emptyset)=(0),$$
which contradicts the hypothesis.

(3). If $R$ is not a field and $I$ is a nonzero proper ideal of $R$, then $I\lr \cap E^0=\emptyset$. Now by using this fact, the statement follows from the part (2).
\end{proof}

Recall that a ring $R$ is said to be left (right) primitive if there exists a faithful simple left (right) $R$-module $M$. Hence, for determining when a Leavitt path algebra $\lr$ is left (or right) primitive we try to determine when a simple and faithful left (or right) $\lr$-module exists. At first, we give the form of simple left $\lr$-modules.

\begin{lemma}\label{lem6.3}
Let $E$ be a graph and let $R$ be a unital commutative ring. If $M$ is a simple left $\lr$-module, then there exist a vertex $v\in E^0$ and a maximal left $\lr$-submodule $I$ of $\lr v$ such that $M\cong \lr v/I$.
\end{lemma}

\begin{proof}
It is well-known that there is a maximal ideal $J$ of $\lr$ such that $M\cong \lr/J$. Take a vertex $v\notin J$. Then we have $J+\lr v=\lr$ and so
$$M\cong \frac{\lr}{J}\cong\frac{J+\lr v}{J}\cong \frac{\lr v}{I}$$
where $I:=J\cap \lr v$ is a maximal left $\lr$-submodule of $\lr v$.
\end{proof}

\begin{lemma}\label{lem6.4}
Let $E$ be a graph containing a closed simple path with no exits and let $R$ be a unital commutative ring. Then $\lr$ is not primitive.
\end{lemma}

\begin{proof}
If $\alpha$ is a closed simple path based at $v$ with no exits, then Equation \ref{2.1} shows that
\begin{align*}
v\lr v&=\mathrm{span}_R\left\{\alpha\beta^*:r(\alpha)=r(\beta)~\mathrm{and}~s(\alpha)=s(\beta)=v\right\}\\
&=\mathrm{span}_R\left\{v,\alpha^m,\alpha^{*n}:m,n\in\mathbb{N}\right\}.
\end{align*}
The latter $R$-algebra is isomorphic to the Leavitt path algebra of the graph
$$\includegraphics[width=3cm]{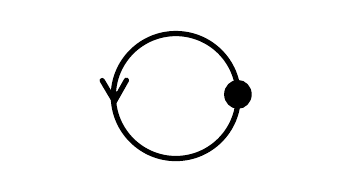}$$
which is isomorphic to the Laurent polynomials ring $R([x,x^{-1}])$ as rings. However, this algebra is not primitive and hence $\lr v$ is not primitive. (Note that a commutative ring is primitive if and only if it is a field.) Now the fact that every corner of a primitive algebra is also primitive yields the result.
\end{proof}

We denote by Mod-$\lr$ the collection of all left $\lr$-modules. As in \cite{ara1}, for $v\in E^0$ we denote\vspace{.2cm}

\leftline{\hspace{.7cm}$S_v:=\{M\in\mathrm{Mod}-\lr: M\cong\frac{\lr v}{I}$}

\rightline{$~\mathrm{for~some~maximal~left~submodule~of~}v\lr\}$.\hspace{.7cm}}\vspace{.2cm}
An standard argument similar to the proof of \cite[Proposition 4.4]{ara3} gives the following lemma.

\begin{lemma}\label{lem6.5}
Let $E$ be a graph and let $R$ be a unital commutative ring. If $v,w\in E^0$ and $\alpha$ is a path from $v$ to $w$, then
\begin{enumerate}[$(1)$]
\item $S_v=S_w$,
\item $\lr v/I\cong \lr w/I\alpha$ for every maximal left $\lr$-submodule $I$ of $\lr v$.
\end{enumerate}
\end{lemma}

\begin{lemma}\label{lem6.6}
Let $E$ be a graph and let $R$ be a unital commutative ring. If $v$ is a vertex with $|s^{-1}(v)|\geq 2$ and $\lr v/I$ is a simple left $\lr$-module for a maximal left $\lr$-submodule $I$ of $\lr v$, then
$$\frac{\lr v}{I}\cong\frac{\lr r(e)}{\lr e}$$
for every edge $e\in s^{-1}(v)$ but probably one.
\end{lemma}

\begin{proof}
For $n\geq 2$, select $e_1,\ldots,e_n\in s^{-1}(v)$ and write $v=e_1e_1^*+\ldots+e_ne_n^*+x$ where $x:=v-\sum_{i=1}^ne_ie_i^*$. Then $I=Iv\subseteq Ie_1e_1^*\oplus\ldots\oplus Ie_ne_n^*\oplus Ix\subseteq \lr v$. Since $I$ is a maximal left $\lr$-submodule of $\lr v$, we have two possibilities.

$\underline{\emph{Case I}}$ : $Ie_1e_1^*\oplus\ldots\oplus Ie_ne_n^*\oplus Ix=\lr v.$\\
Then for every $a\in\lr$, there are $t_i,s\in I$ such that $av=\sum_{i=1}^n t_ie_i e_i^*+sx$, and so by multiplying $e_ie_i^*$ from right hand side we have $ae_ie_i^*=t_ie_ie_i^*\in I$ for all $i$. Thus, for each $i$, $\lr e_ie_i^*=Ie_ie_i^* $, $\lr e_i=Ie_i$, and
$$\frac{\lr v}{I}\cong\frac{\lr r(e_i)}{Ie_i}\cong\frac{\lr r(e_i)}{\lr e_i}$$
by Lemma \ref{lem6.5}(2).

$\underline{\emph{Case II}}$ : $Ie_1e_1^*\oplus\ldots\oplus Ie_ne_n^*\oplus Ix=I.$\\
Then we have
$$\frac{\lr v}{I}\cong\frac{\oplus_{i=1}^n\lr e_ie_i^*\oplus \lr x}{\oplus_{i=1}^nIe_ie_i^* \oplus Ix}\cong\bigoplus_{i=1}^n\frac{\lr e_ie_i^*}{Ie_ie_i^*}\bigoplus\frac{\lr x}{Ix}$$
because $e_ie_i^*$ and $x$ are pairwise orthogonal. The simplicity of $\lr v/ I$ implies that all summands but only one are zero. For such indexes $i$, we have $\lr e_ie_i^*=I e_ie_i^*$, $\lr e_i=I e_i$, and
$$\frac{\lr v}{I}\cong\frac{\lr r(e_i)}{I e_i}\cong\frac{\lr r(e_i)}{\lr e_i}.$$
This completes the proof.
\end{proof}

Now we are in the position to characterize primitive Leavitt path algebras.

\begin{theorem}[See Theorem 4.6 of \cite{ara3}]\label{thm6.7}
Let $E$ be a graph and let $R$ be a unital commutative ring. Then the following are equivalent.
\begin{enumerate}[$(1)$]
\item $\lr$ is left primitive.
\item $\lr$ is right primitive.
\item $R$ is a field and $E$ satisfies Conditions (L) and (MT3).
\item $I\cap J\cap E^0\neq \emptyset$ for every nonzero ideals $I$ and $J$ of $\lr$.
\end{enumerate}
\end{theorem}

\begin{proof}
(3) $\Leftrightarrow$ (4) is Proposition \ref{prop6.2}(3). We only prove (1) $\Leftrightarrow$ (3). The implications (2) $\Leftrightarrow$ (3) may be proved analogously.

(1) $\Rightarrow$ (3). Let $\lr$ be left primitive. Then $\lr$ is a prime ring and so Proposition \ref{prop4.4} yields that $E$ satisfies Condition (MT3). Also, Lemma \ref{lem6.4} implies that $E$ also satisfies Condition (L).

Now we show that $R$ is a field. Let $M$ be a faithful simple left $\lr$-module. Then Lemma \ref{lem6.3} gives that $M\cong \lr v/I$ for some $v\in E^0$ and some maximal left $\lr$-submodule $I$ of $\lr v$. Hence for every $a\in I$ we have $aM=0$ and so $a=0$ by the faithfulness. Thus $I=(0)$ and $M=\lr v$. Now if $R$ has a proper nonzero ideal $J$, then $JM=J\lr v$ is a proper and nonzero left $\lr$-submodule of $M$, a contradiction. This concludes that $R$ is a field.

(3) $\Rightarrow$ (1). Suppose that $R$ is a field and $E$ satisfies Conditions (L) and (MT3). If $P_l(E)\neq \emptyset$ and $v\in P_l(E)$, then Proposition \ref{prop5.5} implies that $M=\lr v$ is a minimal left ideal and so $M$ is a simple left $\lr$-module. (Note that every left $\lr$-submodule of $\lr v$ is of the form $Jv$ where $J$ is a left ideal of $\lr$.) $M$ is also faithful. Indeed, if $aM=a\lr v=0$ for some $a\in \lr$, then we have $a=0$ because $\lr$ is prime by Proposition \ref{prop4.4}. So, $M$ is a desired left $\lr$-module and therefore $\lr$ is left primitive.

Now assume that $P_l(E)=\emptyset$. Since $E$ satisfies Condition (L), there is a vertex $v\in E^0$ with $|s^{-1}(v)|\geq 2$. For every $w\in E^0$, Condition (MT3) yields that there is a vertex $y\in E^0$ such that $v,w\geq y$. Hence, we have $S_v=S_y=S_w$ by Lemma \ref{lem6.5}(1).

If $M$ is a simple left $\lr$-module, then $M\cong \lr w/I$ for some $w\in E^0$ and some maximal left $\lr$-submodule $I$ of $\lr w$. But since $S_v=S_w$, $M$ is isomorphic to $\lr v/J$ for some maximal left $\lr$-submodule $J$ of $\lr v$. Therefore, by Lemma \ref{lem6.6}, we have
\begin{align*}
&\left\{M:M\mathrm{~is~a~simple~left~}\lr\mathrm{-module}\right\}\\
=&\left\{M:M\in S_v\right\}\\
=&\left\{\frac{\lr r(e)}{\lr e}:e\in s^{-1}(v)\mathrm{~and~}\frac{\lr r(e)}{\lr e}~\mathrm{is~simple~}\right\}.
\end{align*}
Note that Lemma \ref{lem6.6} gives that the later set contains at most two nonisomorphic left $\lr$-modules. Thus, the Jacobson radical of $\lr$ is
\begin{align*}
J(\lr)&=\bigcap_{M~\mathrm{left~simple}}\mathrm{Ann} (M)\\
&=\bigcap_{s(e)=v,~\frac{\lr r(e)}{\lr e}~\mathrm{simple}}\mathrm{Ann}\left(\frac{\lr r(e)}{\lr e}\right),
\end{align*}
and since $J(\lr)=0$ (see [3, Proposition 6.3]), by applying Proposition \ref{prop6.2}(3) we get that $\mathrm{Ann}(\lr r(e)/\lr e)=0$ for some simple left $\lr$-module $\lr r(e)/\lr e$. Therefore, $\lr$ is left primitive.
\end{proof}

\begin{remark}
We know that if a commutative ring $R$ is primitive, then it is a field. If $R$ is a unital commutative ring, then $R$ can be considered as the Leavitt path algebra of a single vertex. However, Theorem \ref{thm6.7} ensures that if $\lr$ is primitive for some graph $E$, then $R$ is a field. Furthermore, by comparing Theorem \ref{thm6.7} and Proposition \ref{prop4.4}, we see that a prime Leavitt path algebra may be nonprimitive duo to its coefficients ring. For example the Toeplitz algebra with coefficients in $\Z$ which is the Leavitt path algebra of the graph
$$\includegraphics[width=4cm]{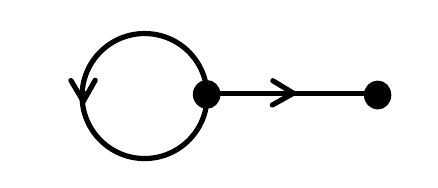}$$
is a nonprimitive prime ring because $\Z$ is not a field.

Now, as for prime ideals in Section 4, we can characterize primitive graded ideals of Leavitt path algebras by applying Theorem \ref{thm6.7}. Recall that an ideal $I$ of a ring $R$ is said to be primitive if the quotient $R/I$ is a primitive ring. Duo to Theorem \ref{thm6.7}(3) we assume that $R$ is a filed. It is well-known that every primitive ideal is also prime. Hence, we look up in the graded prime ideals of $\lr$ described in Proposition \ref{prop4.4} for determining primitive graded ideals. We denote by $M_\gamma(E)$ the set of maximal tails $M$ in $E$ such that the subgraph $(M,r^{-1}(M),r,s)$ satisfies Condition (L).
\end{remark}

\begin{proposition}\label{prop6.9}
Let $E$ be a graph and let $R$ be a field. Then
$$\left\{I(\Omega(M),B_{\Omega(M)}),I(\Omega(v),B_{\Omega(v)}\setminus\{v\}):M\in M_\gamma(E), v\in BV(E)\right\}$$
is the set of primitive graded ideals of $\L_R(E)$.
\end{proposition}

\begin{proof}
First, recall from Lemma \ref{lem4.6} that if $M\in M(E)$, then $\Omega(M)=E^0\setminus M$ and so we have $E/(\Omega(M),B_{\Omega(M)})=(M,r^{-1}(M),r,s)$. Thus an ideal $I(\Omega(M),B_{\Omega(M)})$ is primitive if and only if the quotient
$$\frac{L_R(E)}{I(\Omega(M),B_{\Omega(M)})}\cong L_R\left(E/(\Omega(M),B_{\Omega(M)})\right)=L_R(M,r^{-1}(M),r,s)$$
is primitive. However, this is equivalent to $M\in M_\gamma(E)$ by Theorem \ref{thm6.7}. Furthermore, if $v\in BV(E)$, the quotient graph $E/(\Omega(v),B_{\Omega(v)}\setminus\{v\})$ contains the sink $v'$. But for every vertex $w\in E^0\setminus\Omega(v)$ we have $w\geq v$. Thus for every vertex $w$ in  $E/(\Omega(v),B_{\Omega(v)}\setminus\{v\})$, there is a path from $w$ to $v'$. This yields that  $E/(\Omega(v),B_{\Omega(v)}\setminus\{v\})$ satisfies Conditions (L) and (MT3). Therefore, Theorem \ref{thm6.7} implies that $L_R(E/(\Omega(v),B_{\Omega(v)}\setminus\{v\}))\cong \lr/I(\Omega(v),B_{\Omega(v)}\setminus\{v\})$ is primitive and so $I(\Omega(v),B_{\Omega(v)}\setminus\{v\})$ is a primitive ideal.
\end{proof}

\begin{corollary}\label{cor6.10}
Let $E$ be a graph satisfying Condition (K) and let $R$ be a field. Then $$\left\{I(\Omega(M),B_{\Omega(M)}),I(\Omega(v),B_{\Omega(v)}\setminus\{v\}):M\in M(E), v\in BV(E)\right\}$$
is the set of all primitive ideals of $L_R(E)$. In particular, the set of prime ideals and the set of primitive ideals of $\lr$ coincide.
\end{corollary}

\begin{proof}
Lemma \ref{lem3.15} implies that every quotient graph of $E$ satisfies Conditions (L) and so we have $M_\gamma(E)=M(E)$. On the other hand, by Corollary \ref{cor3.19}, every ideal of $\lr$ is of the form of $I(H,S)$ for some admissible pair $(H,S)$. However, Proposition \ref{prop6.9} characterize all primitive ideals of this form. \end{proof}

\begin{remark}
Since any quotient of $L_R(E)$ by a graded basic ideal belongs to the class of Leavitt path algebras (Theorem \ref{thm3.10}(3)), Theorem \ref{thm6.7} shows that when $R$ is not a field, $\lr$ contains no primitive graded basic ideals. However, in this case, $\lr$ may have some primitive ideals. For example, consider $\Z$ as the Leavitt path algebra of a single vertex $v$ with the coefficients ring $\Z$. Then the ideal $I=\langle 2v\rangle$ is left and right primitive because $\Z/I\cong \Z_2$ is a field.
\end{remark}


\end{document}